\documentclass[a4paper, 11pt]{article}  
\usepackage{a4wide}
\usepackage{authblk}
\usepackage{hyperref}

\usepackage{amsthm}

\newtheorem{thm}{Theorem}[section]
\newtheorem{assum}[thm]{Assumption}
\newtheorem{prop}[thm]{Proposition}
\newtheorem{defn}[thm]{Definition}

\newtheorem{rem}[thm]{Remark}

\newtheorem{lem}[thm]{Lemma}

\usepackage{float}
\usepackage{tikz}
\usepackage{dsfont}
\usepackage{graphicx}      
\usepackage[numbers]{natbib}
\usepackage{natbib}        
\usepackage{systeme}
\usepackage{amsmath}
\usepackage{amssymb}
\allowdisplaybreaks

\DeclareMathOperator{\rank}{rank}

\newcommand{\finproof}{\hfill\small$\blacksquare$}

\title{Stabilization of a Chain of Three Hyperbolic PDEs using a Time-Delay Representation\thanks{This project received funding from the Agence Nationale de la Recherche via grant PANOPLY ANR-23-CE48-0001-01.}}

\author[1]{Adam Braun}
\author[1]{Jean Auriol}
\author[1]{Lucas Brivadis}
\affil[1]{Université Paris-Saclay, CNRS, CentraleSupélec, Laboratoire des Signaux et Systèmes, 91190, Gif-sur-Yvette, France. Emails:
        {\tt\small adam.braun@centralesupelec.fr; jean.auriol@centralesupelec.fr; lucas.brivadis@centralesupelec.fr}}%

\begin{document}

\maketitle






\begin{abstract}                


This paper addresses the stabilization of a chain system consisting of three hyperbolic Partial Differential Equations (PDEs). The system is reformulated into a pure transport system of equations via an invertible backstepping transformation. Using the method of characteristics and exploiting the inherent cascade structure of the chain, the stabilization problem is reduced to that of an associated Integral Difference Equation (IDE).
A dynamic controller is designed for the IDE, whose gains are computed by solving a system of Fredholm-type integral equations.
This approach provides a systematic framework for achieving exponential stabilization of the chain of hyperbolic PDEs.
\end{abstract}

\textbf{Keywords.} Networked Systems, PDE, IDE, Stabilization


\section{Introduction}
The controllability and stabilization of networks of hyperbolic Partial Differential Equations (PDEs) represent a fundamental area of research with broad applications in engineering and science. Hyperbolic PDEs govern systems where wave propagation and transport phenomena are predominant, such as fluid dynamics~\citep{HayatShang2021}, electrical networks \citep{Arrillaga1998}, and mechanical vibrations~\citep{BastinCoron2016,auriol2022comparing,yu2023traffic}. In most applications, these systems are often interconnected in a network structure, where the dynamic behavior of each subsystem is influenced by its neighbors through boundary interactions. Given a control input acting on the boundary of one subsystem in the network, the stabilization of the entire network is a problem of significant interest in various fields. For instance, it arises in the context of traffic flows where ramp metering systems can be used to control freeway interconnections~\citep{Yu2022}.


There has been significant attention given to cascaded interconnections of hyperbolic PDE systems, possibly coupled with Ordinary Differential Equations (ODEs) in ODE-PDE-ODE configurations, and most existing constructive control strategies being based on the backstepping approach as evidenced in~\citep{deutscher2021backstepping,Deutscher2018,wang2020delay,irscheid2023output,auriol2023robustification}. A recent approach based on infinite dimensional dynamics extension for the backstepping control of hyperbolic PDE systems has been developed in~\citep{GEHRING2025112032}. The backstepping approach has recently proven effective in designing stabilizing controllers for interconnections of hyperbolic PDEs with chain structures.  For such configurations, backstepping-based controllers were initially developed for interconnected scalar subsystems~\citep{deutscher2021backstepping,auriol2020output,Redaud2021}. Extensions to non-scalar subsystems have been proposed in~\citep{auriol2024output}. Interestingly, in most contributions for such interconnected configurations, the design of a stabilizing control law relies on reformulating the interconnection as a \emph{time-delay system}. However, these results also usually assume that control input is positioned at one end of the chain. Although such a configuration covers a wide range of applications, such as drilling pipes or UAV-cable-payload structures, there are several situations in which the actuator is located at an arbitrary node of the chain. For instance, when developing traffic control strategies on vast road networks, the actuator (ramp metering) can be located at a crossroad. In the case of two coupled hyperbolic PDEs with a control input applied at the junction, a stabilizing control law was proposed in~\citep{Redaud2022} using Fredholm transformations. While this approach is well-suited for the two-PDE configuration, it exhibits certain limitations when extended to chains comprising more than two PDEs. More recently, a novel method addressing some of these limitations was developed in~\citep[Chapter 9]{AuriolHDR} and~\citep{AuriolLCSS}. Using backstepping transforms, the system is rewritten as an Integral Difference Equation, and the control law is designed in terms of integrals that depend on the state and input histories.

The present paper addresses the stabilization of three scalar hyperbolic PDE systems interconnected in a chain structure. The control input is located between the first and second systems. The method previously developed in \citep[Chapter 9]{AuriolHDR} cannot be directly applied, rendering the case of three PDEs not a straightforward extension of the two-PDE case.
This work represents a crucial step towards the stabilization of a chain composed of $N$ hyperbolic PDEs with an input located at different possible node of the chain.
Inspired by the previous contributions on the stabilization of two interconnected systems~\citep[Chapter 9]{AuriolHDR},~\citep{Redaud2022} actuated at the junction, we first apply a backstepping transformation 
This enables us to map the original system onto an intermediate target system composed of transport equations with integral couplings at their boundaries. The stability properties of this target system are equivalent to those of an integral difference equation (IDE) incorporating both pointwise and distributed delays. The stabilization of this IDE will be achieved through an adaptation of the method proposed in~\citep{AuriolLCSS}. The control law will have an auto-regressive structure~\citep{ABNT-IJRNC-2024}. In particular, the control input gains are derived from the solutions of Fredholm equations. 

 The structure of this article follows the outlined strategy. In Section~\ref{Problem formulation}, we define the problem and present our assumptions. Section~\ref{Time-Delay representation} details the steps we take to reformulate the problem as the stabilization of an IDE with multiple state delays. Finally, in Section~\ref{controller design}, we design a stabilizing state-feedback law for the IDE. 
 \section{Problem formulation}
\label{Problem formulation}
\subsection{System under consideration}
Let $L^2(0, 1)$ be the space of real valued square integrable functions over $(0, 1)$,
and let $H^1(0, 1)=\{f\in L^2(0, 1): f'\in L^2(0, 1)\}$.
In this paper, we consider a chain of three subsystems interconnected through their boundaries as schematically represented in Figure~\ref{system_draw}.
The control input~$U$ is located at the boundary between the first and the second subsystem and takes real values. Every subsystem $i$ is represented by a set of PDEs. For  all $i\in \{1,2,3\},$ the states $u_i, v_i$ verify
\begin{equation}
\label{system_edp_origin}
\begin{aligned} 
    &\partial_t u_i(t,x) + \lambda_i \partial_x u_i(t,x) = \sigma^{+}_i(x) v_i(t,x),\\
     &\partial_t v_i(t,x) - \mu_i \partial_x v_i(t,x) = \sigma^{-}_i(x) u_i(t,x),\\
&u_i(t,0) = q_{ii}v_{i}(t,0) +q_{i i-1}u_{i-1}(t,1)\\
&v_i(t,1) =\mathds{1}_{i=1}U(t) +\rho_{ii}u_i(t,1) + \rho_{ii+1}v_{i+1}(t,0)
\end{aligned}
\end{equation}
with $q_{10} = \rho_{34} = 0$, $t\geq 0$ and $x\in [0, 1]$. The velocities $\lambda_i$ and $\mu_i$ are positive real constants. The in-domain coupling terms $\sigma^{+}_i$ and $\sigma^{-}_i$ are continuous functions defined on $[0, 1]$ with real values.
The boundary coupling coefficients $q_{ij}$, $\rho_{ij}$
are real constants. Finally, we define $\tau_i$ as the total transport time associated with the subsystem $i$: \begin{equation}\label{tau_i def}\tau_i = \frac{1}{\lambda_i} + \frac{1}{\mu_i}.\end{equation}
  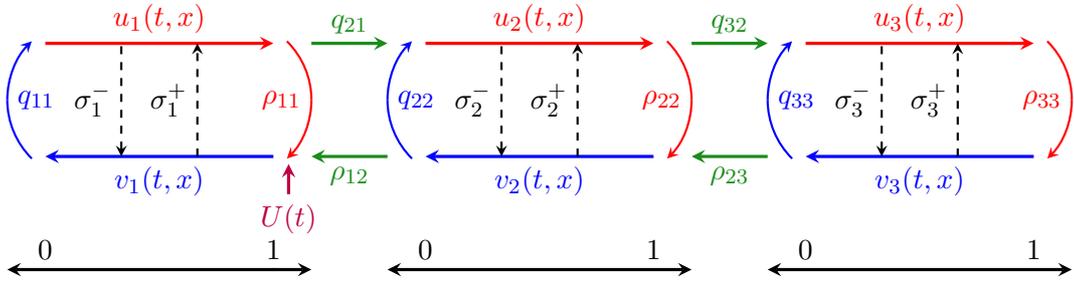
\begin{figure}[htb]%

\begin{center}

 \scalebox{1}{

\begin{tikzpicture}



\draw [>=stealth,->,red,very thick] (0,0) -- (3,0);

\draw [red] (1.5,0) node[above]{$u_1(t,x)$};

\draw [>=stealth,<-,blue,very thick] (0,-1.5) -- (3,-1.5);

\draw [blue] (1.5,-1.5) node[below]{$v_1(t,x)$};


\draw [>=stealth,<-,dashed, thick] (1,-1.5) -- (1,0);

\draw(1,-0.75) node[left]{$\sigma^{-}_1$};

\draw [>=stealth,->,dashed, thick] (2,-1.5) -- (2,0);

\draw(2,-0.75) node[left]{$\sigma^{+}_1$};


\draw [blue,>=stealth, thick](-0.5,-0.75) arc (-180:-135:1.1);

\draw [blue,>=stealth,->, thick](-0.5,-0.75) arc (-180:-225:1.1);

\draw [blue] (-0.5,-0.75) node[right]{$q_{11}$};
\draw [>=stealth,<-,purple,very thick] (3.2,-1.6) -- (3.2,-2);

\draw [purple] (3.2,-2) node[below]{$U(t)$};

\draw [red,>=stealth, thick](3.5,-0.75) arc (0:45:1.1);

\draw [red,>=stealth,->, thick](3.5,-0.75) arc (0:-45:1.1);

\draw [red] (3.5,-0.75) node[left]{$\rho_{11}$};


\draw [>=stealth,->,green!50!black!90,very thick] (3.5,0) -- (4.5,0);

\draw [color=green!50!black!90] (4,0) node[above]{$q_{21}$};

\draw [>=stealth,<-,green!50!black!90,very thick] (3.5,-1.5) -- (4.5,-1.5);

\draw [color=green!50!black!90] (4,-1.5) node[below]{$\rho_{12}$};




\draw [>=stealth,->,red,very thick] (5,0) -- (8,0);

\draw [red] (6.5,0) node[above]{$u_{2}(t,x)$};

\draw [>=stealth,<-,blue,very thick] (5,-1.5) -- (8,-1.5);

\draw [blue] (6.5,-1.5) node[below]{$v_{2}(t,x)$};


\draw [>=stealth,<-,dashed, thick] (6,-1.5) -- (6,0);

\draw(6,-0.75) node[left]{$\sigma^{-}_{2}$};

\draw [>=stealth,->,dashed, thick] (7,-1.5) -- (7,0);

\draw(7,-0.75) node[left]{$\sigma^{+}_{2}$};


\draw [blue,>=stealth, thick](4.5,-0.75) arc (-180:-135:1.1);

\draw [blue,>=stealth,->, thick](4.5,-0.75) arc (-180:-225:1.1);
\draw [red] (8.5,-0.75) node[left]{$\rho_{22}$};
\draw [blue] (4.5,-0.75) node[right]{$q_{22}$};


\draw [red,>=stealth, thick](8.5,-0.75) arc (0:45:1.1);

\draw [red,>=stealth,->, thick](8.5,-0.75) arc (0:-45:1.1);



\draw [>=stealth,->,green!50!black!90,very thick] (8.5,0) -- (9.5,0);

\draw [color=green!50!black!90] (9,0) node[above]{$q_{32}$};

\draw [>=stealth,<-,green!50!black!90,very thick] (8.5,-1.5) -- (9.5,-1.5);

\draw [color=green!50!black!90] (9,-1.5) node[below]{$\rho_{23}$};



\draw [>=stealth,->,red,very thick] (10,0) -- (13,0);

\draw [red] (11.5,0) node[above]{$u_3(t,x)$};

\draw [>=stealth,<-,blue,very thick] (10,-1.5) -- (13,-1.5);

\draw [blue] (11.5,-1.5) node[below]{$v_3(t,x)$};


\draw [>=stealth,<-,dashed, thick] (11,-1.5) -- (11,0);

\draw(11,-0.75) node[left]{$\sigma^{-}_3$};

\draw [>=stealth,->,dashed, thick] (12,-1.5) -- (12,0);

\draw(12,-0.75) node[left]{$\sigma^{+}_3$};


\draw [blue,>=stealth, thick](9.5,-0.75) arc (-180:-135:1.1);

\draw [blue,>=stealth,->, thick](9.5,-0.75) arc (-180:-225:1.1);

\draw [blue] (9.5,-0.75) node[right]{$q_{33}$};

\draw [red,>=stealth, thick](13.5,-0.75) arc (0:45:1.1);

\draw [red,>=stealth,->, thick](13.5,-0.75) arc (0:-45:1.1);

\draw [red] (13.5,-0.75) node[left]{$\rho_{33}$};


\draw [>=stealth,<->,very thick] (4.5,-3) -- (8.5,-3);

\draw [>=stealth,<->,very thick] (9.5,-3) -- (13.5,-3);

\draw (8,-3) node[above]{1};

\draw (10,-3) node[above]{0};

\draw (5,-3) node[above]{0};

\draw [very thick] (3,-3) -- (3,-3);

\draw (3,-3) node[above]{1};

\draw (13,-3) node[above]{1};

\draw [>=stealth,<->,very thick] (-0.5,-3) -- (3.5,-3);

\draw (0,-3) node[above]{0};


\end{tikzpicture}}
\caption{Schematic representation of system~\eqref{system_edp_origin}}
\label{system_draw}
\end{center}

\end{figure}


According to \citep[Appendix A]{BastinCoron2016}, the open-loop system \eqref{system_edp_origin} is well-posed: for any input $U\in C^0([0,+\infty), \mathbb R)$ and any initial condition $(u^0,v^0)\in L^2([0,1], \mathbb{R}^3)\times L^2([0,1], \mathbb{R}^3)$, it admits a unique corresponding solution in  $C^0([0,+\infty), L^2([0,1], \mathbb{R}^3))\times C^0([0,+\infty), L^2([0,1], \mathbb{R}^3))$.

In this paper, we want to exponentially stabilize~\eqref{system_edp_origin} in the sense of the $L^2$-norm. More precisely, we recall the following definition of exponential stability.
\begin{defn} System~\eqref{system_edp_origin} with $U\equiv0$ is exponentially stable if there exist $\omega>0$ and $K\geq 1$ such that, for all $(u_0, v_0)\in L^2([0,1],\mathbb{R}^3)\times L^2([0,1],\mathbb{R}^3)$, the unique solution to \eqref{system_edp_origin} is such that, for all $t\geq 0$,
        $\|(u(t), v(t))\|_{L^2} \leq K e^{-\omega t}\|(u_0, v_0)\|_{L^2} $.
    \end{defn}
\subsection{Structural Assumptions} 
\label{assumptions}
The following assumptions are necessary to design a stabilizing control law.
\begin{assum} \label{Assum_stab_OL}
If $\sigma^{+}_i$ and $\sigma^{-}_i$ are equal to zero for all $i\in \{1,2,3\}$,
then
system~\eqref{system_edp_origin} with $U\equiv 0$ is exponentially stable.
\end{assum}
This assumption is analogous to the one stated in~\citep[Assumption 5]{Redaud2021} and relates to a (delay-) robustness condition. More precisely, without this condition, it can be shown~\citep{AuriolHDR} that the open-loop system has an infinite number of unstable poles. This would prevent the design of stabilizing feedback control law robust to input delays~\citep{logemann1996conditions}.
As Assumption~\ref{Assum_stab_OL} implies that the boundary coupling terms are dissipative, the presence of the in-domain coupling terms is the main difficulty for the stabilization of~\eqref{system_edp_origin}.
\begin{assum} \label{non_zero_coeff_nes}
The boundary coupling coefficients $q_{21}$, $q_{32}$ 
are non-zero.
\end{assum}
This assumption guarantees that the control input
can act on the second and third subsystems. Obviously Assumption~\ref{non_zero_coeff_nes} is not necessary if the two last subsystems are already exponentially stable. However, this configuration is of minor interest.
\begin{assum}\label{non_zero_coeff_non_nes}
The boundary coupling coefficient $q_{11}$ is non-zero.
\end{assum}
When \( q_{11} = 0 \), the control law influences the second subsystem solely through the in-domain coupling term \( \sigma_1^+ \). Assumption~\ref{non_zero_coeff_non_nes} induces some conservatism but is currently required by the approach of the present article, which exploits the action of the control through the boundary couplings.
To these three assumptions, an additional controllability condition will be introduced in Section 4, as further background is required to formally state it.
\section{Time-Delay Representation}
\label{Time-Delay representation}
To achieve exponential stabilization, we first use a backstepping transformation to map the system~\eqref{system_edp_origin} to an intermediate target system designed to eliminate the in-domain coupling. Subsequently, the system is reduced to a set of transport equations, the in-domain coupling terms having been replaced by integral terms acting at the boundaries of the different subsystems. 
\subsection{Intermediate target system}
Let us consider the backstepping transformation
$\mathcal{F}$ that maps $(u_i, v_i)_{1\leq i\leq 3}\in H^1([0,1],\mathbb{R}^6)$ to $(\alpha_i, 
\beta_i)_{1\leq i\leq 3}\in H^1([0,1],\mathbb{R}^6)$ and defined for all $x\in [0,1]$ by:
\begin{align}
    \mathcal{\alpha}_1(x) &= u_1(x)
    -\int_0^x \begin{pmatrix}
        K_1^{uu}(x, y)\\ K_1^{uv} (x, y)
    \end{pmatrix}^\top\begin{pmatrix}
        u_1(y)\\ v_1(y) 
    \end{pmatrix}dy 
    \label{alpha_1},\\
    \beta_1(x) &= v_1(x)  -\int_0^x \begin{pmatrix}
        K_1^{vu}(x, y)\\ K_1^{vv} (x, y)
    \end{pmatrix}^\top\begin{pmatrix}
        u_1(y)\\ v_1(y) 
    \end{pmatrix}dy, \label{beta_1}
\\
    \mathcal{\alpha}_2(x) &= u_2(x) -\int_x^1 \begin{pmatrix}
        K_2^{uu}(x, y)\\ K_2^{uv} (x, y)
    \end{pmatrix}^\top\begin{pmatrix}
        u_2(y)\\ v_2(y) 
    \end{pmatrix}dy\label{alpha_2}-\int_0^1 \begin{pmatrix}Q^{\mathcal{\alpha}}(x,y)\\R^{\mathcal{\alpha}}(x,y) \end{pmatrix}^\top \begin{pmatrix}
        u_3(y)\\
        v_3(y)
    \end{pmatrix}dy, \\
    \beta_2(x) &= v_2(x)  -\int_x^1 \begin{pmatrix}
        K_2^{vu}(x, y)\\ K_2^{vv} (x, y)
    \end{pmatrix}^\top\begin{pmatrix}
        u_2(y)\\ v_2(y) 
    \end{pmatrix}dy \label{beta_2} -\int_0^1 \begin{pmatrix} Q^{\beta}(x,y)\\R^{\beta}(x,y) \end{pmatrix}^\top \begin{pmatrix}
        u_3(y)\\
        v_3(y)
    \end{pmatrix}dy,
\\ 
    \mathcal{\alpha}_3(x) &= u_3(x) -\int_x^1 \begin{pmatrix}
        K_3^{uu}(x, y)\\ K_3^{uv} (x, y)
    \end{pmatrix}^\top\begin{pmatrix}
        u_3(y)\\ v_3(y) 
    \end{pmatrix}dy \label{alpha_3},\\
    \beta_3(x) &= v_3(x)  -\int_x^1 \begin{pmatrix}
        K_3^{vu}(x, y)\\ K_3^{vv} (x, y)
    \end{pmatrix}^\top\begin{pmatrix}
        u_3(y)\\ v_3(y) 
    \end{pmatrix}dy, \label{beta_3}
\end{align}
where $K_i^{uu}$, $K_i^{uv}$, $K_i^{vu}$, $K_i^{vv}$, $Q^{\alpha}$, $Q^{\beta}$, $R^{\alpha}$ and $R^{\beta}$
are piecewise continuous kernels to be defined, and
${}^\top$ denotes the transpose operator.
The map $\mathcal{F}$ is known to be continuously invertible as a combination of Volterra and affine transforms (see, e.g., \citep{Yoshida1960}). Moreover, its inverse $\mathcal{F}^{-1}:(\mathcal{\alpha}_i, 
\beta_i)_{1\leq i\leq 3}\mapsto (u_i, v_i)_{1\leq i\leq 3}$ takes the form:
\begin{align*}
    u_1(x) &= \mathcal{\alpha}_1(x) -\int_0^x \begin{pmatrix}L_1^{\alpha \alpha}(x,y)\\L_1^{\alpha \beta}(x,y) \end{pmatrix}^\top \begin{pmatrix}
        \mathcal{\alpha}_1(y)\\
        \beta_1(y)
    \end{pmatrix} dy,\\
    v_1(x)& = \beta_1(x)  -\int_0^x \begin{pmatrix}L_1^{\beta \alpha}(x,y)\\L_1^{\beta \beta}(x,y) \end{pmatrix}^\top \begin{pmatrix}
        \mathcal{\alpha}_1(y)\\
        \beta_1(y)
    \end{pmatrix}dy,\\
    u_2(x)& = \mathcal{\alpha}_2(x) -\int_x^1  \begin{pmatrix}L_2^{\alpha \alpha}(x,y)\\L_2^{\alpha \beta}(x,y) \end{pmatrix}^\top \begin{pmatrix}
        \mathcal{\alpha}_2(y)\\
        \beta_2(y)
    \end{pmatrix} dy-\int_0^1\begin{pmatrix}
        S^{u}(x,y)\\T^{u}(x,y)
    \end{pmatrix}^T \begin{pmatrix}
        \mathcal{\alpha}_3(y)\\
        \beta_3(y)
    \end{pmatrix}dy,\\
    v_2(x) &= \beta_2(x)  -\int_x^1 \begin{pmatrix}L_2^{\beta \alpha}(x,y)\\L_2^{\beta \beta}(x,y) \end{pmatrix}^\top \begin{pmatrix}
        \mathcal{\alpha}_2(y)\\
        \beta_2(y)
    \end{pmatrix}dy-\int_0^1 \begin{pmatrix}
        S^{v}(x,y)\\T^{v}(x,y)
    \end{pmatrix}^T \begin{pmatrix}
        \mathcal{\alpha}_3(y)\\
        \beta_3(y)
    \end{pmatrix}dy,\\
    u_3(x) &= \mathcal{\alpha}_3(x) -\int_x^1 \begin{pmatrix}L_3^{\alpha \alpha}(x,y)\\L_3^{\alpha \beta}(x,y) \end{pmatrix}^\top \begin{pmatrix}
        \mathcal{\alpha}_3(y)\\
        \beta_3(y)
    \end{pmatrix} dy,\\
    v_3(x)&= \beta_3(x)  -\int_x^1 \begin{pmatrix}L_3^{\beta \alpha}(x,y)\\L_3^{\beta \beta}(x,y) \end{pmatrix}^\top \begin{pmatrix}
        \mathcal{\alpha}_3(y)\\
        \beta_3(y)
    \end{pmatrix}dy,
\end{align*}
for some piecewise continuous kernels $L_i^{\cdot \cdot}$, $S^{u}$, $S^{v}$, $T^{u}$ and $T^{v}$.
Choosing the kernels according to the equations given in Appendix~\ref{Backstepping Kernels}, one can readily show that $\mathcal F$ maps the original system \eqref{system_edp_origin} into the following system composed of pure transport equations:
\begin{equation}\label{transport_intermediate_system}
\begin{aligned}
    \partial_t \alpha_i(t,x) + \lambda_i \partial_x \alpha_i(t,x) &= 0, \\
    \partial_t \beta_i(t,x) - \mu_i \partial_x \beta_i(t,x) &=0,
\end{aligned}
\quad \forall i\in \{1,2,3\}.
\end{equation}
Moreover, we obtain the associated boundary conditions: 
\begin{equation}
\begin{aligned}
 \alpha_1(t,0) &= q_{11}\beta_1(t,0),\\
 \beta_1(t,1) &= \bar{U}(t),\\
\alpha_2(t,0) &= q_{22}\beta_2(t,0) + q_{21}\alpha_1(t,1) + \int_0^1 \sum_{i=1}^3 \alpha_i(t,y) P_{2i}(y)+\beta_i(t,y)W_{2i}dy, \\
 \beta_2(t,1) &= \rho_{22}\alpha_2(t,1) + \rho_{23}\beta_3(t,0)+ \int_0^1 \alpha_3(t,y) J_{23}(y)+\beta_3(t,y)C_{23}dy, \\
 \alpha_3(t,0)&= q_{32} \alpha_2(t,1) + q_{33}\beta_3(t,0) + \int_0^1 \alpha_3(t,y)P_{33}(y) + \beta_3(t,y)W_{33}(y)dy,
 \\
 \beta_3(t,1) &= \rho_{33}\alpha_3(t,1),
\end{aligned}\label{boundary_cond_target_system}
\end{equation}
where we have replaced $U(t)$ with a new control input $\bar U(t)$ by choosing
\begin{align*}
\bar{U}(t) &=U(t)-\rho_{12}\beta_2(t,0) + \rho_{11}\alpha_1(t,1)-\int_0^1 \sum_{i=1}^3 \alpha_i(t,y) J_{1i}(y)+\beta_i(t,y)C_{1i}dy,
 \end{align*}

and
\begin{align*}
 P_{21}(y)&=-q_{21}L_1^{\alpha \alpha}(1,y),~W_{21}(y) = -q_{21} L_1^{\alpha \beta}(1,y),\\
 P_{22}(y)&=L_2^{\alpha \alpha}(0,y) -q_{22}L_2^{\beta \alpha}(0,y),\\
  P_{23}(y) &=-q_{22}S^v(0,y) + S^u(0,y), \\
    W_{23}(y)&=-q_{22}T^v(0,y) + T^u(0,y),\\
   W_{22}(y)&=L_2^{\alpha \beta}(0,y)dy-q_{22}L_2^{\beta \beta}(0,y),\\
    P_{33}(y) &= -q_{32}S^u(1,y) - q_{33}L_3^{\beta \alpha}(0,y)+L_3^{\alpha \alpha}(0,y),\\
    W_{33}(y) &=-q_{32}T^u(1,y) -q_{33}L_3^{\beta \beta}(0,y) + L_3^{\alpha \beta}(0,y),\\
    J_{11}(y)&= L_1^{\beta \alpha}(1,y) -\rho_{11}L_1^{\alpha \alpha}(1,y),\\
    C_{11}(y)&= L_1^{\beta \beta}(1,y) -\rho_{11}L_1^{\alpha \beta}(1,y),\\
    J_{12}(y)&=-\rho_{12}L_2^{\beta \alpha}(0,y) ,~C_{12}(y)= -\rho_{12}L_2^{\beta \beta}(1,y),\\
     J_{13}(y)&= -\rho_{12}S^{v}(0,y),~C_{13}(y)= -\rho_{12}T^{v}(0,y))\\
    J_{23}(y)&=S^v(1,y) - \rho_{23} L_3^{\beta \alpha}(0,y) -\rho_{22} S^u(1,y),\\
    C_{23}(y)&=T^v(1,y) - \rho_{23}L_3^{\beta \beta}(0,y) -\rho_{22}T^u(1,y).
\end{align*}
\begin{rem}
Due to the cancellation of the boundary reflection terms, the control law $U(t)$ may not be robust to delays but it is possible to robustify it by applying a well-tuned low pass filter \citep{auriol2023robustification}.
\end{rem}
\subsection{Integral Difference Equation}
In this section, and throughout the remainder of the article, we assume that the initial conditions \( (u_0, v_0) \) belong to \( H^1([0,1], \mathbb{R}^3) \times H^1([0,1], \mathbb{R}^3) \) and satisfy the zero-order boundary conditions specified in~\eqref{system_edp_origin}. Consequently, by applying the results from~\citep[Appendix A]{BastinCoron2016}, the original open loop system~\eqref{system_edp_origin} admits a unique solution in \( C^0([0,+\infty), H^1([0,1], \mathbb{R}^3)) \times C^0([0,+\infty), H^1([0,1], \mathbb{R}^3)) \). Since the control to be designed will be admissible and continuous in time, the same reference guarantees the well-posedness of the resulting closed-loop system.

In order to stabilize the transport system \eqref{transport_intermediate_system}, we
propose to rewrite it as an IDE. We define
$x_2(t) = \alpha_2(t,0)$, $\bar{x}_2(t) =\beta_2(t,1)$, and $x_3(t) = \alpha_3(t,0)$ as the new states, and rewrite the boundary conditions~\eqref{boundary_cond_target_system} using the method of characteristics. 
 For a function $\phi: [-\tau, \infty) \mapsto \mathbb{R}^n$, its partial trajectory $\phi_{[t]}$ is defined by $\phi_{[t]}(\theta)=\phi(t+\theta), -\tau \leq \theta \leq 0$.
Let us denote $L^2_{\tau}:=L^2(-\tau,0)$.

Recall that $\tau_i$, as defined in~\eqref{tau_i def}, denotes the total transport time corresponding to subsystem $i$.
We suppose $\tau_2 \leq \tau_3$ as the method in the other case is identical.
For all $t>\tau_1+\tau_2+\tau_3$, we obtain:
\begin{align}
x_2(t) &=q_{22}\bar{x}_2(t-1/ \mu_2) + q_{21}q_{11}\bar{U}(t-\tau_1) +\int_0^{\tau_1}\bar{U}(t-\nu)M(\nu)d\nu + \int_0^{\tau_3}x_2(t-\nu)H^{22}(\nu)d\nu \nonumber\\
&~+\int_0^{\tau_3}\bar{x}_2(t-\nu)H^{2 \bar{2}}(\nu)d\nu +\int_0^{\tau_3} x_3(t-\nu)H^{23}(\nu)d\nu \label{x2},\\
\bar{x}_2(t) &= \rho_{23}\rho_{33}x_3(t-\tau_3) + \rho_{22}x_2(t-1/ \lambda_2) + \int_0^{\tau_3}H^{\bar{2} 3}(\nu)x_3(t-\nu) d\nu \label{x_2barre},\\
x_3(t) &= q_{32}x_2(t-\frac{1}{\lambda_2})+q_{33}\rho_{33}x_3(t-\tau_3) + \int_0^{\tau_3}H^{33}(\nu)x_3(t-\nu)d\nu,\label{x3}
\end{align}
where,
\begin{align*}
H^{33}(\nu)&= \mathds{1}_{[0, 1/\lambda_3]}(\nu)P_{33}(\nu \lambda_3)\lambda_3
\\
&~+ \mathds{1}_{[1/\lambda_3, \tau_3]}(\nu)W_{33}(1-\nu \mu_3 + \frac{\mu_3}{\lambda_3})
\mu_3,\\
H^{\bar{2} 3}(\nu) &=  \mathds{1}_{[0, 1/\lambda_3]}(\nu)\lambda_3 J_{23}(\nu \lambda_3) \\
&~+  \mathds{1}_{[1/\lambda_3, \tau_3]}(\nu)\mu_3 \rho_{33} C_{23}(1-\nu \mu_3 + \frac{\mu_3}{\lambda_3}), \\
H^{23}(\nu)&=  \mathds{1}_{[0, 1/\lambda_3]}(\nu)\lambda_3P_{23}(\nu \lambda_3) \\
&~+ \mathds{1}_{[1/\lambda_3, \tau_3]}(\nu)\mu_3 \rho_{33}W_{23}(1-\nu \mu_3 + \frac{\mu_3}{\lambda_3}),\\
M(\nu)&=   \mathds{1}_{[0, 1/\mu_1]}(\nu)\mu_1W_{21}( 1-\mu_1 \nu) \\
&~+  \mathds{1}_{[1/\mu_1, \tau_1]}(\nu)\lambda_1 q_{11}P_{21}( \nu \lambda_1 -\frac{\lambda_1}{\mu_1})),\\
H^{22}(\nu) &= \mathds{1}_{[0,1/\lambda_2]}(\nu)\lambda_2P_{22}(\nu\lambda_2),\\
H^{2 \bar{2}}(\nu) &= \mathds{1}_{[0,1/\mu_2]}(\nu)\mu_2W_{22}( 1-\mu_2 \nu),
\end{align*}
and $\mathds{1}_I(\nu) := \begin{cases}1~\text{if}~\nu \in I\\
0~\text{if not}\end{cases}$ for $I$ an interval of $\mathbb{R}$.

The well-posedness of the open-loop system~\eqref{x2}--\eqref{x3} is established in~\cite[Chapter 9]{halebook}. We now recall the definition of exponential stability for the open-loop system~\eqref{x2}--\eqref{x3}.
\begin{defn}
    System~\eqref{x2}-\eqref{x3} with $\bar U\equiv0$ is exponentially stable if there exist
    $\omega_0 > 0,$ and $K_0 \geq 1,$ such that for all $(x_2^0, \bar{x}^0_2, x^0_3) \in \left(L^2_{\tau_2+\tau_3}\right)^3$, the unique corresponding solution to~\eqref{x2}-\eqref{x3} 
    is such that, for all $t\geq0$,    $$\|x_{2{[t]}},\bar{x}_{2{[t]}}, x_{3[t]}\|_{\left(L^2_{\tau_2+\tau_3}\right)^3} 
    \leq K_0\mathrm{e}^{-\omega_0 t}\|x_2^0, \bar{x}^0_2, x^0_3\|_{\left(L^2_{\tau_2+\tau_3}\right)^3}$$
\end{defn}
By exploiting the method of characteristics, it was shown in
\citep[Theorem 6.1.3]{AuriolHDR}
that the exponential stability of~\eqref{x2}-\eqref{x3} is equivalent to the exponential stability of system~\eqref{system_edp_origin}. Therefore, we now focus on the exponential stabilization of the IDE~\eqref{x2}-\eqref{x3}.

As $q^{32}\neq 0$ due to Assumption~\ref{non_zero_coeff_nes}, the exponential stability of $x_3$ implies the exponential stability of $x_2$.
Hence, multiplying equation~\eqref{x_2barre} by $q_{32}$ gives us a relation between $\bar{x}_2$ and $x_3$.
We then multiply equation~\eqref{x2}by $q_{32}$ and do a change of variable ($\Tilde{t} = t-\frac{1}{\lambda_2}$, we shall keep the variable $t$ in the following), to get an equation that only depends on $x := x_3$ :
\begin{align}\label{IDE_formulation}
x(t) &= a_{3} x(t-\tau_3) + a_{2}x(t-\tau_2) + a_{23}x(t-\tau_2-\tau_3) + b\bar{U}(t-\tau_1-\frac{1}{\lambda_2}) \nonumber\\
&~+ \int_{0}^{\tau_1+\frac{1}{\lambda_2}} \Tilde{M}(\nu)\bar{U}(t-\nu)d\nu + \int_0^{\tau_3 + \tau_2}N(\nu)x(t-\nu),\end{align}
where $a_3=q_{33}\rho_{33},a_2 = q_{22}\rho_{22}, a_{23} =q_{22}(q_{32}\rho_{23}\rho_{33}-q_{33}\rho_{33}\rho_{22})$, and $b= q_{32}q_{21}q_{11}\neq 0$ (according to Assumptions~\ref{non_zero_coeff_nes} and \ref{non_zero_coeff_non_nes}),
$\Tilde{M}:= q_{32}M(\cdot - \frac{1}{\lambda_2})\mathds{1}_{[ \frac{1}{\lambda_2}, \tau_1 + \frac{1}{\lambda_2}]}(\cdot)$ and 
\begin{align*}
    &N(\nu)=\mathds{1}_{[0,\tau_3]}(\nu)\left(H^{33}(\nu)+H^{22}(\nu)-\int_0^{\nu}H^{33}(\eta)H^{22}(\nu-\eta)d\eta\right)\\
    &-\mathds{1}_{[\tau_3,2\tau_3]}(\nu)\left(q_{33}\rho_{33}H^{22}(\nu-\tau_3)-\int_{\nu-\tau_3}^{\tau_3}H^{33}(\eta)H^{22}(\nu-\eta)d\eta\right)\\
    &+\mathds{1}_{[\tau_2,\tau_2+\tau_3]}(\nu)q_{22}\left(q_{33}H^{\bar{2}3}(\nu-\tau_2)-\rho_{22}H^{33}(\nu-\tau_2)\right)\\
    &+\mathds{1}_{[\frac{1}{\lambda_2},\tau_3+\frac{1}{\lambda_2}]}(\nu)\Bigg(q_{32}H^{23}\left(\nu-\frac{1}{\lambda_2}\right)\\
    &+\int_0^{\nu-\frac{1}{\lambda_2}}H^{2\bar{2}}(\nu-\frac{1}{\lambda_2}-\eta)(q_{33}H^{\bar{2}3}(\eta)-\rho_{22}H^{33}(\eta))d\eta\Bigg)\\
    &+\mathds{1}_{[\frac{1}{\lambda_2}+\tau_3,2\tau_3+\frac{1}{\lambda_2}]}(\nu)\Bigg((q_{32}\rho_{23}\rho_{33}-q_{33}\rho_{33}\rho_{22})H^{2\bar{2}}(\nu-\frac{1}{\lambda_2}-\tau_3)\\
    &+\int_{\nu-\tau_3-\frac{1}{\lambda_2}}^{\tau_3}H^{2\bar{2}}(\nu-\frac{1}{\lambda_2}-\tau_3-\eta)(q_{33}H^{\bar{2}3}(\eta)-\rho_{22}H^{33}(\eta))d\eta\Bigg).
\end{align*}
The kernels $N$ and $\Tilde{M}$ are piecewise continuous functions respectively defined on $[0,\tau_2+\tau_3]$ and $[0, \tau_1+ \frac{1}{\lambda_2}]$.

After the backstepping transformation of Section 2, and the time-delay reformulation of Section 3,  we have therefore shown that the original problem of exponential stabilization of the chain of hyperbolic PDEs~\eqref{system_edp_origin} is equivalent to the problem of
exponential stabilization of
the IDE \eqref{IDE_formulation}.
We focus on the design of a stabilizing controller in the next section.

Moreover, we suppose $\tau_1 + \frac{1}{\lambda_2}>\tau_2 + \tau_3 $, i.e., the delay on the input is longer than the longest delay on the state.
The case $\tau_1 + \frac{1}{\lambda_2}\leq \tau_2 + \tau_3$ could be covered with the same method or by artificially delaying the control by taking $\bar{U}(t) = \hat{U}(t - \tau_2 -\tau_3 +\tau_1 +\frac{1}{\lambda_2})$.
\section{Design of a State-Feedback Controller}
\label{controller design}
Our objective is now to stabilize the IDE~\eqref{IDE_formulation} in the sense of the $L^2$ norm as this will imply the exponential stability of the PDE system. 
Due to Assumption~\ref{Assum_stab_OL}, we have that the principal part of \eqref{IDE_formulation} (i.e., the system defined by $x(t) = a_{3} x(t-\tau_3) + a_{2}x(t-\tau_2) + a_{23}x(t-\tau_2-\tau_3)$) is exponentially stable in the sense of the $L^2$-norm. (see \citep[Theorem 6.1.3]{AuriolHDR}).  Hence, the main difficulty in the stabilization of~\eqref{IDE_formulation} consists of dealing with the integral term with the kernel $N$, which can possibly destabilize the principal part.
We rely on the following condition:
\begin{assum}\label{spectral_cont_retard}
     $\forall s\in \mathbb{C},
    \rank (F_0(s), F_1(s)) = 1$, where
    \begin{align*}F_0(s) &=1- a_3e^{-\tau_3 s} - a_2e^{-\tau_2 s}-a_{23}e^{-(\tau_2 +\tau_3) s}-\int_0^{\tau_2 +\tau_3}N(\nu)e^{-\nu s}d\nu,\end{align*}
    and $$F_1(s)=be^{-(\tau_1 + \frac{1}{\lambda_2})s}+\int_{0}^{\tau_1 + \frac{1}{\lambda_2}}\Tilde{M}(\nu)e^{-\nu s} d\nu.$$
\end{assum}
\begin{rem}This assumption is thought to follow from a spectral controllability condition, as in \citep[Chapter 9]{AuriolHDR}. However, in the case of three interconnected hyperbolic systems, the computations required to establish possible connections with the controllability properties of the PDE~\citep{fattorini1966some} become exceedingly complex due to the intricate structure of the kernel $N$. Furthermore, this assumption is not necessary as it should be possible to use a weaker condition, for all $s\in \mathbb{C}$ such that $\Re(s) \geq 0$, $\rank(F_0(s), F_1(s) = 1$. But again the analysis would become more difficult as it is not directly possible to show the invertibility of the Fredholm operator~\eqref{tau_def}.
 \end{rem}
Inspired by~\citep{AuriolLCSS}, we consider the dynamic control law defined by
\begin{align}\label{U_somme_retard_x}
\bar{U}(t) &= \int_0^{\tau_1 + \frac{1}{\lambda_2}} g(\nu)\bar{U}(t-\nu)d \nu + \int_0^{\tau_2 +\tau_3} f(\nu)x(t-\nu)d\nu  \end{align}
where the functions $g$,  and $f = \mathds{1}_{[0, \tau_2)}f_2 +
\mathds{1}_{[\tau_2, \tau_3)}f_3
+\mathds{1}_{[\tau_3, \tau_2+\tau_3]}f_{23}$ are controller gains yet to be determined. This control law has an auto-regressive structure~\citep{ABNT-IJRNC-2024} as it depends on past values of itself.

The stability analysis will be conducted in the Laplace domain. To ease the reading, the Laplace transform of a function of time $y(t)$ will be denoted $y(s)$. We will actually show the exponential stability of the vector $(x(t),U(t))$, as this extended state can be seen as the solution of an extended IDE.
Taking the Laplace transform of equations~\eqref{IDE_formulation} and~\eqref{U_somme_retard_x} we get for $s \in \mathbb{C}$
$$\begin{pmatrix}
    0\\
    0
\end{pmatrix} = Q(s) \begin{pmatrix}
    x(s)\\
    U(s)
\end{pmatrix},$$ 
with $Q$ defined by,
\begin{equation*}\label{Q_laplace_eq_retard}Q(s) = \begin{pmatrix}
    F_0(s) & -F_1(s) \\
    \int_0^{\tau_2+\tau_3} f(\nu)e^{-\nu s}d\nu  &
    1- \int_0^{\tau_1 + \frac{1}{\lambda_2}} g(\nu)e^{-\nu s}d\nu
\end{pmatrix}\end{equation*}
The associated characteristic equation is defined by:
\begin{align}\label{characteristic equation}
    0 = \det(Q(s)).
\end{align}
To investigate the stability of the closed-loop system, we shall use the following lemma.
\begin{lem} \label{stable_lemma}
(\citep[Theorem 3.5]{halebook} and \citep{henry1974linear}).
The closed-loop system~\eqref{IDE_formulation}-\eqref{U_somme_retard_x} is exponentially stable if and only if there exists~$\eta_0>0$ such that all complex solutions $s$ of its associated characteristic equation \eqref{characteristic equation} satisfy~$\text{Re}(s)<-\eta_0$.
\end{lem}
The design objective is then to find $g,f$ such that the characteristic equation~\eqref{characteristic equation} is reduced to the characteristic equation of the principal part, which is already exponentially stable.

Let $s$ be a complex solution of the characteristic equation~\eqref{characteristic equation}. Then, 
\begin{align*}
 0&=1-a_3 e^{-\tau_3 s} -  a_2e^{-\tau_2 s}-a_{23}e^{-(\tau_2 +\tau_3) s}
-\int_0^{\tau_2 +\tau_3}N(\nu)e^{-\nu s}d\nu 
    -\int_0^{\tau_1 + \frac{1}{\lambda_2}} g(\nu)e^{-\nu s}d\nu\\ 
    &~+\int_{\tau_3}^{\tau_1 + \tau_3 + \frac{1}{\lambda_2}}a_3 g(\nu-\tau_3)e^{-\nu s}d\nu
   +\int_{\tau_2}^{\tau_1 + \tau_2 + \frac{1}{\lambda_2}}a_2 g(\nu-\tau_2)e^{-\nu s}d\nu \\ &~+\int_{\tau_3+\tau_2}^{\tau_1 + \tau_3 +\tau_2+ \frac{1}{\lambda_2}}a_{23} g(\nu-\tau_3-\tau_2)e^{-\nu s}d\nu+\int_0^{\tau_1 + \frac{1}{\lambda_2}} \int_0^{\tau_2 + \tau_2}N(\nu)g(\eta)e^{-(\nu + \eta)s}d\eta d\nu\\
    &~-\int_{\tau_1 + \frac{1}{\lambda_2} }^{\tau_1 + \frac{1}{\lambda_2} + \tau_3+\tau_2 }b f(\nu-\tau_1 - \frac{1}{\lambda_2})e^{-\nu s}d\nu-\int_{0}^{\tau_2 +\tau_3} \int_{0}^{\tau_1 + \frac{1}{\lambda_2}}\Tilde{M}(\nu)f(\eta)e^{-(\nu + \eta)s}d\eta d\nu \\
    &= 1-a_3 e^{-\tau_3 s} -  a_2e^{-\tau_2 s}-a_{23}e^{-(\tau_2 +\tau_3) s}-\int_0^{\tau_1 + \frac{1}{\lambda_2}} (I_1^{g,f}(\nu) + N(\nu)\mathds{1}_{[0, \tau_3 + \tau_2]}(\nu)e^{-\nu s}d\nu
    \\ &~-\int_{\tau_1+\frac{1}{\lambda_2}}^{\tau_2+\tau_1+\frac{1}{\lambda_2}} I_2^{g,f}(\nu-\tau_1-\frac{1}{\lambda_2})e^{-\nu s}d\nu
-\int_{\tau_2+\tau_1+\frac{1}{\lambda_2}}^{\tau_3+\tau_1+\frac{1}{\lambda_2}} I_3^{g,f}(\nu-\tau_1-\frac{1}{\lambda_2})e^{-\nu s}d\nu
    \\&~-\int_{\tau_3+\tau_1+\frac{1}{\lambda_2}}^{\tau_3 +\tau_2+\tau_1+\frac{1}{\lambda_2}} I_{23}^{g,f}(\nu-\tau_1-\frac{1}{\lambda_2})e^{-\nu s}d\nu,
\end{align*}
with $a_{\tau_2}=a_2$, $a_{\tau_3} = a_3$, $a_{\tau_2+\tau_3} = a_{23}$, and
\begin{align*}
    I_1^{g,f}(\nu) &=g(\nu) - \sum_{\delta \in \{\tau_2, \tau_3, \tau_2 +\tau_3\}} a_{\delta} g(\nu-\delta)\mathds{1}_{(\delta,\tau_1 + \frac{1}{\lambda_2}]}(\nu) \\
    &~+ \mathds{1}_{[0, \tau_2 +\tau_3]}(\nu) \int_0^{\nu}f(\eta)\Tilde{M}(\nu-\eta)- g(\eta)N(\nu-\eta)d\eta \\
    &~+\mathds{1}_{(\tau_2 +\tau_3,\tau_1 + \frac{1}{\lambda_2}]}(\nu)\Bigg(\int_0^{\tau_2 +\tau_3} f(\eta)\Tilde{M}(\nu-\eta)d\eta - \int_{\nu-\tau_3-\tau_2}^\nu g(\eta)N(\nu-\eta)d\eta \Bigg),\\
    I_2^{g,f}(\nu) &=bf_2(\nu) -  \sum_{\delta \in \{\tau_2, \tau_3, \tau_2 +\tau_3\}} a_{\delta} g(\nu+\tau_1 +\frac{1}{\lambda_2}-\delta)\mathds{1}_{[0,\delta]}(\nu)
    \\
    &~+\int_{\nu}^{\tau_2+\tau_3}f(\eta)\Tilde{M}(\nu-\eta)d\eta
    -\int_{\nu-\tau_3-\tau_2+\tau_1+\frac{1}{\lambda_2}}^{\tau_1+\frac{1}{\lambda_2}} g(\eta)N(\nu-\eta)d\eta,\\
    I_3^{g,f}(\nu) &=bf_3(\nu) -  \sum_{\delta \in \{\tau_3, \tau_2 +\tau_3\}} a_{\delta} g(\nu+\tau_1 +\frac{1}{\lambda_2}-\delta)\mathds{1}_{[\tau_2,\delta]}(\nu)
    \\
    &~+\int_{\nu}^{\tau_2+\tau_3}f(\eta)\Tilde{M}(\nu-\eta)d\eta
    -\int_{\nu-\tau_3-\tau_2+\tau_1+\frac{1}{\lambda_2}}^{\tau_1+\frac{1}{\lambda_2}} g(\eta)N(\nu-\eta)d\eta,\\
    I_{23}^{g,f}(\nu) &=bf_{23}(\nu) -  a_{23} g(\nu+\tau_1 +\frac{1}{\lambda_2}-\tau_2-\tau_3)\mathds{1}_{[\tau_3,\tau_2 +\tau_3]}(\nu)
    +\int_{\nu}^{\tau_2+\tau_3}f(\eta)\Tilde{M}(\nu-\eta)d\eta
    \\&~-\int_{\nu-\tau_3-\tau_2+\tau_1+\frac{1}{\lambda_2}}^{\tau_1+\frac{1}{\lambda_2}} g(\eta)N(\nu-\eta)d\eta,
\end{align*}
where we have computed the double integral terms using Proposition \ref{prop:fubini} in Appendix~\ref{fubini}.

Choosing $g$ and $f$ such that $I_1^{g,f} = -N\mathds{1}_{[0, \tau_3 + \tau_2]}$, $I_2^{g,f}=0$, $I_3^{g,f}=0$, $I_{23}^{g,f}=0$ ensures the solutions of the characteristic equation~\eqref{characteristic equation}  are the same as the ones of the principal part which is exponentially stable due to Assumption~\ref{Assum_stab_OL}. We then prove the exponential stability of the closed-loop system~\eqref{IDE_formulation}-\eqref{U_somme_retard_x}  using \citep{halebook}[Theorem 3.5] and \citep{henry1974linear}.
The problem is reformulated as finding $g,f$ such that $\mathcal{T}(g,f) = (-N\mathds{1}_{[0, \tau_3 + \tau_2]},0)$
with 
\begin{align}\mathcal{T}~:&~X\longrightarrow X\label{tau_def}\\ \nonumber
&(g,f)^T \mapsto (I_1^{g,f}, I_2^{g,f}, I_3^{g,f}, I_{23}^{g,f})^T
\end{align}
and $X:=L^2(0,\tau_1 + \frac{1}{\lambda_2})\times L^2(0,\tau_2)\times L^2(\tau_2,\tau_3)\times L^2(\tau_3,\tau_2 + \tau_3)$.
We will establish the invertibility of $\mathcal{T}$ by applying the following lemma.
\begin{lem}\label{Coron lemma}
\citep[Lemma 2.2]{Coron2016}.
    Consider two linear operators $\mathcal{A},\mathcal{B}$, such that $D(\mathcal{A})=D(\mathcal{B}) \subset X $. Consider an operator $\mathcal{T}:  X \rightarrow X$ of the form $\mathcal{T} = D + K$ with $D$ an invertible operator on $X$ and $K$ a compact operator on $X$. Assume that 
\begin{enumerate}
    \item $\ker(\mathcal{T}) \subset D(\mathcal{A})$,
    \item $\ker(\mathcal{T}) \subset \ker(\mathcal{B})$,
    \item $\forall h \in \ker(\mathcal{T}),~\mathcal{T} \mathcal{A} h=0 $,
    \item $\forall s \in \mathbb{C}$, $\ker (s\text{Id}-\mathcal{A}) \cap \ker (\mathcal{B})=\{0\}$.
\end{enumerate}
Then, the operator $\mathcal{T}$ is invertible.
\end{lem}
\begin{thm} Under Assumptions~\ref{non_zero_coeff_nes},~\ref{non_zero_coeff_non_nes},~\ref{spectral_cont_retard}. The operator \label{tau_invertible}
    $\mathcal{T}$ defined by equation~\eqref{tau_def} is invertible. Hence, there exists a unique pair $(g,f)\in X$ such that $$\mathcal{T}(g,f) = (-N\mathds{1}_{[0, \tau_3 + \tau_2]},0).$$
\end{thm}
\begin{proof}
The integral part of $\mathcal{T}$ is compact because the integral kernels are $L^2$ functions \citep[Chapter 6]{brezisfunctionnalanalysis}, and the remaining part is invertible because of its triangular shape.
    We will verify the conditions to apply Lemma~\ref{Coron lemma} using the following operators
    \begin{align*}
        \mathcal{A}~:&~D(\mathcal{A})\longrightarrow X\\
        &\begin{pmatrix}
            \psi\\ \phi_2\\ \phi_3 \\\phi_{23}
        \end{pmatrix} \mapsto \begin{pmatrix}\psi' +\phi_2(0)\Tilde{M}\\ \phi_2' +\phi_2(0)N\\\phi_3' +\phi_2(0)N\\ \phi_{23}' +\phi_2(0)N  \end{pmatrix} 
    \end{align*}
 \begin{align*}
        \mathcal{B}~:&~D(\mathcal{A})\longrightarrow X\\
        &(
            \psi, \phi_2, \phi_3,\phi_{23}
       )^T \mapsto \psi(0),
    \end{align*}
$D(\mathcal{A}) = \{(\psi,\phi_2, \phi_3,\phi_{23}):\psi(\tau_1 +\frac{1}{\lambda_2}) = b\phi_2(0), \phi_2(\tau_2) - a_2 \phi_2(0) - \phi_{3}(\tau_2) = 0,
\phi_3(\tau_3) - a_3 \phi_2(0) - \phi_{23}(\tau_3) = 0,
\phi_{23}(\tau_3+\tau_2) =a_{23}\phi_2(0)\}
\cap H^1(0,\tau_1 + \frac{1}{\lambda_2})\times H^1(0,\tau_2)\times H^1(\tau_2,\tau_3)\times H^1(\tau_3,\tau_2 + \tau_3)$.\\
Let $g,f \in \ker \mathcal{T}$. Because $I_1^{g,f}(0) = 0\Leftrightarrow g(0) =0$, requirement $(2)$ of Lemma~\ref{Coron lemma} is proved. From $I_1^{g,f}(\tau_1 +\frac{1}{\lambda_2}) =0$ and $I_2^{g,f}(0) =0$, we deduce $bf_2(0) = g(\tau_1 +\frac{1}{\lambda_2})$. Using $I_{23}^{g,f}(\tau_2 + \tau_3) =0$ we deduce $bf_{23}(\tau_2+\tau_3) = a_{23}g(\tau_1 + \frac{1}{\lambda_2})$. The previous equality and $b\neq 0$ imply $f_{23}(\tau_1 + \tau_3) = a_{23}f_2(0)$. The equations $I_3^{g,f}(\tau_3)=0=I_{23}^{g,f}(\tau_3)$ imply $bf_3(\tau_3) - a_3g(\tau_1 + \frac{1}{\lambda_2})=bf_{23}(\tau_3)$. Again, using $b\neq0$ and $bf_2(0) = g(\tau_1 + \frac{1}{\lambda_2}$ we deduce $f_3(\tau_3) - a_3 f_2(0)=f_{23}(\tau_3)$. Similarly, evaluating $I_2^{g,f}$ and $I_3^{g,f}$ at $\nu =\tau_2$ implies $f_2(\tau_2)-a_2f_2(0)=f_3(\tau_2)$. Hence we have proved requirement $(1)$ of Lemma~\ref{Coron lemma}. 

Let us prove the third requirement of Lemma~\ref{Coron lemma}. Due to space restriction, we only consider the case where $N$ and $\Tilde{M}$ are continuously differentiable functions. The case in which both functions have a finite number of discontinuities follows the same approach. Following the same steps as in~\citep{AuriolLCSS}, we differentiate each equation in $\mathcal{T}(g,f) = 0$, except for the first one. Using this differentiation along with the condition $(g,f) \in D(\mathcal{A})$, we conclude that the corresponding equation in $\mathcal{T} \mathcal{A} (g,f)$ is identically zero after applying integration by parts. 
 For the first line, we apply the same method on each interval $(0, \tau_2),(\tau_2, \tau_3), (\tau_3, \tau_2+ \tau_3),(\tau_2 +\tau_3, \tau_1 +\frac{1}{\lambda_2})$.
 
 Now let us show the last requirement of Lemma~\ref{Coron lemma} using Assumption~\ref{spectral_cont_retard}. Let $s\in \mathbb{C}$, and $(g,f)\in\ker (s\text{Id}-\mathcal{A}) \cap \ker (\mathcal{B})$. As a consequence we have for $\nu\in (0, \tau_1 + \frac{1}{\lambda_2})$
$$sg(\nu)=g'(\nu)+f_2(0)\Tilde{M}(\nu),$$ 
and $g(0) =0$.
Therefore by the Duhamel formula \begin{equation}\label{g}g(\nu) = -\int_0^\nu e^{s(\nu-\eta)} f_2(0)\Tilde{M}(\eta)d\eta.\end{equation}
For $\nu= \tau_1 + \frac{1}{\lambda_2}$ and using $g(\tau_1 + \frac{1}{\lambda_2}) = bf_2(0)$ we find
$$F_1(\nu) f_2(0) = 0.$$
At this point, we aim to show that $F_0(s) f_2(0) = 0$. Using Assumption~\ref{spectral_cont_retard} we could conclude $f_2(0)=0$.\\
For all $k\in\{2,3,23\}$, we have
\begin{equation*} \label{cauchy ode}sf_k(\nu) = f'_k(\nu) + f_2(0)N(\nu).\end{equation*}
Hence by the Duhamel formula
\begin{align}f_2(\nu) &= e^{s\nu}f_2(0) - \int_{0}^\nu e^{s(\nu-\eta)}N(\eta)f_2(0)d\eta, \label{f2}\\
f_3(\nu) &= e^{s(\nu-\tau_2)}f_3(\tau_2) - \int_{\tau_2}^\nu e^{s(\nu-\eta)}N(\eta)f_2(0)d\eta, \label{f3}\\
f_{23}(\nu) &= e^{s(\nu-\tau_3)}f_{23}(\tau_3) - \int_{\tau_3}^\nu e^{s(\nu-\eta)}N(\eta)f_2(0)d\eta. \label{f23}\end{align}
Evaluating each of the previous equations in $\nu =\delta $ and, multiplying it by $-e^{-s\delta}$, ($\delta = \tau_2$ for the first equation, $\delta =\tau_3$ for the second, $\delta = \tau_2 + \tau_3$ for the third) and summing all three equations implies: 
\begin{align*}0&=f_2(0)- e^{-\tau_2 s}(f_2(\tau_2) -f_3(\tau_2)-e^{-\tau_3 s}(f_3(\tau_3)-f_{23}(\tau_3)) -e^{-(\tau_2 +\tau_3) s}f_{23}(\tau_2 +\tau_3) 
\\&~-\int_0^{\tau_2 + \tau_3} N(\nu) e^{-\nu s}d\nu f_2(0).\end{align*}
Now, using $(g,f) \in D(\mathcal{A})$ we deduce:
$$F_0(s)f_2(0) = 0.$$
Assumption~\ref{spectral_cont_retard} implies $f_2(0) = 0$, hence $f_2 \equiv 0$ using\eqref{f2}. By Formula~\eqref{g}, $f_2(0) = 0$ also implies $g\equiv 0$. Because $(g,f) \in D(\mathcal{A})$, $f_2 \equiv  0$ implies $f_3(\tau_2) =f_2(\tau_2)=0$  therefore $f_3 \equiv 0$ using\eqref{f3}. As $(g,f) \in D(\mathcal{A})$ also implies $f_{23}(\tau_3) = f_3(\tau_3) =0$ together with~\eqref{f23}, we deduce $f_{23}\equiv 0$.

We can apply Lemma~\ref{Coron lemma}. Thus, $\mathcal{T}$ is invertible. \finproof
\end{proof}
Combining Lemma~\ref{stable_lemma} and Theorem~\ref{tau_invertible}, we obtain:
\begin{thm}
Under the structural assumptions introduced in Section~\ref{assumptions} and Assumption~\ref{spectral_cont_retard},
the control law defined by
    \begin{align} \label{final control}U(t) &=\int_0^{\tau_1 + \frac{1}{\lambda_2}} g(\nu)\bar{U}(t-\nu)d \nu + \int_0^{\tau_2+\tau_3} f(\nu)x(t-\nu)d\nu \nonumber
-  \rho_{12}\beta_2(t,0) + \rho_{11}\alpha_1(t,1) \nonumber\\
 &~- \int_0^1 \sum_{i=1}^3 \alpha_i(t,y) R_{1i}(y)+\beta_i(t,y)Q_{1i}dy,\end{align} with $g$ and $f$ given by Theorem~\ref{tau_invertible} is such that the closed loop system~\eqref{system_edp_origin}-\eqref{final control} is exponentially stable.
\end{thm}

\section{Conclusion}
In this paper, we extended the method presented in \citep[Chapter 9]{AuriolHDR} to stabilize a chain of three hyperbolic PDE systems with the input located between the first and the second subsystem. Our approach directly uses the cascade structure of the chain. We believe that this method can be generalized to the case of a chain composed of $N$ subsystems with an input located
between the first and second subsystem. For example, the backstepping transformation~\eqref{alpha_1}-\eqref{beta_3} can be adapted by adding in the transformation of the system $i$, affine terms of all the systems located at the right of the system $i$. In the future, we will consider the stabilization of a chain composed of $N$ subsystems, for which the control input can be located at any arbitrary node. We will also focus on the case of multiple control inputs located at several nodes of the chain. This will pave the path towards more complex network configurations.
\bibliographystyle{abbrv}
\bibliography{references}

                                                   






\appendix
 \section{Backstepping Kernels}\label{Backstepping Kernels}
In this appendix, we give the kernel equations associated with the backstepping transformation~\eqref{alpha_1}-\eqref{beta_3}.
 For all $i\in \{1,2,3\}$, we have:
\begin{align*}
-\lambda_i \partial_x K_i^{uu}(x,y) - \lambda_i\partial_y K_i^{uu}(x,y) &= K_i^{uv}(x,y)\sigma_i^{-}(y) ,\\
-\lambda_i \partial_x K_i^{uv}(x,y) + \mu_i\partial_y K_i^{uv}(x,y) &= K_i^{uu}(x,y)\sigma_i^{+}(y), \\
\mu_i \partial_x K_i^{vu}(x,y) - \lambda_i\partial_y K_i^{vu}(x,y)& = K_i^{vv}(x,y)\sigma_i^{-}(y),\\
\mu_i \partial_x K_i^{vv}(x,y) + \mu_i\partial_y K_i^{vv}(x,y) &= K_i^{vu}(x,y)\sigma_i^{+}(y),
\end{align*}
and 
\begin{align*}
\mu_2 \partial_x R^{\beta}(x,y) + \mu_3 \partial_y R^{\beta}(x,y) &= Q^{\beta}(x,y)\sigma_3^{+}(y),\\
\mu_2 \partial_x Q^{\beta}(x,y) - \lambda_3 \partial_y Q^{\beta}(x,y) &= R^{\beta}(x,y)\sigma_3^{-}(y),\\
-\lambda_2 \partial_x R^{\alpha}(x,y) + \mu_3 \partial_y R^{\alpha}(x,y) &= Q^{\alpha}(x,y)\sigma_3^{+}(y),\\
-\lambda_2 \partial_x Q^{\alpha}(x,y) - \lambda_3 \partial_y Q^{\alpha}(x,y) &= R^{\alpha}(x,y)\sigma_3^{-}(y).
\end{align*}
With the following boundary conditions,
\begin{align*}
K_1^{uv}(x,x) &= \frac{\sigma_1^{+}(x)}{\mu_1 + \lambda_1}, ~
K_1^{vu}(x,x) = \frac{\sigma_1^{-}(x)}{\mu_1 + \lambda_1},\\
K_2^{uv}(x,x) &= \frac{\sigma_2^{+}(x)}{-\mu_2 - \lambda_2},~
K_2^{vu}(x,x) = \frac{\sigma_2^{-}(x)}{\mu_1 + \lambda_1},\\
K_3^{uv}(x,x) &= \frac{\sigma_3^{+}(x)}{-\mu_3- \lambda_3}, ~
K_3^{vu}(x,x) = \frac{\sigma_3^{-}(x)}{-\mu_3- \lambda_3},
\end{align*}
\begin{align*}
K_1^{uv}(x,0)& = \frac{\lambda_1 q_{11}}{\mu_1}K_1^{uu}(x,0), ~
K_1^{vv}(x,0) = \frac{\lambda_1 q_{11}}{\mu_1}K_1^{vu}(x,0), \\
K_2^{uu}(x,1) &= \frac{\mu_2\rho_{22}}{\lambda_2}K_2^{uv}(x,1) + \frac{\lambda_3 q_{32}}{\lambda_2} Q^{\alpha}(x,0),\\
R^{\alpha}(x,0)& = \frac{\mu_2\rho_{23}}{\mu_3}K_2^{uv}(x,1) + \frac{\lambda_3q_{33}}{\mu_3}Q^{\alpha}(x,0),\\
Q^{\alpha}(x,1) &= \frac{\mu_3\rho_{33}}{\lambda_3}R^{\alpha}(x,1),\\
R^{\beta}(x,0)& = \frac{\mu_2\rho_{23}}{\mu_3}K_2^{vv}(x,1) + \frac{\lambda_3q_{33}}{\mu_3}Q^{\beta}(x,0),\\
K_2^{vu}(x,1) &= \frac{\mu_2\rho_{22}}{\lambda_2}K_2^{vv}(x,1) + \frac{\lambda_3 q_{32}}{\lambda_2} Q^{\beta}(x,0),\\
Q^{\beta}(x,1) &= \frac{\mu_3\rho_{33}}{\lambda_3}R^{\beta}(x,1),\\
Q^{\alpha}(1,y) &= R^{\alpha}(1,y)=Q^{\beta}(1,y) =R^{\beta}(1,y)=0,\\
K_3^{uu}(x,1) &= \frac{\mu_3 \rho_{33}}{\lambda_3} K_3^{uv}(x,1),~
K_3^{vu}(x,1) = \frac{\mu_3 \rho_{33}}{\lambda_3} K_3^{vv}(x,1).
\end{align*}
Following the proof of \citep[Lemma 8]{AuriolPietri}, such kernels indeed exist, are unique and are piecewise continuous functions.

\section{A computational trick}
\label{fubini}
 Using Fubini's theorem and a change of variable, one can prove:
 \begin{prop}\label{prop:fubini}
    For $z \in L^2(0,\theta)$ and $P \in L^2(0, \delta)$, $s\in \mathbb{C}$,
    \begin{align*}
    &\int_0^{\theta}\int_0^{\delta}z(\eta)P(\nu)e^{-(\nu+\eta)s}d\nu d\eta
    =\int_0^{min(\theta, \delta)}\int_0^{\nu}z(\eta)P(\nu-\eta)d\eta e^{-\nu s}d\nu \\
    &~+
    \int_{min(\theta, \delta)}^{max(\theta, \delta)}\int_{max(\nu-\delta, 0)}^{min(\nu, \theta)}z(\eta)P(\nu-\eta)d\eta e^{-\nu s}d\nu+\int_{max(\theta, \delta)}^{\theta + \delta} \int_{\nu-\delta}^{\theta}z(\eta)P(\nu-\eta)d\eta e^{-\nu s}d\nu.
    \end{align*}
\end{prop}
\end{document}